\documentclass[12pt,reqno]{amsart} % amsart

\usepackage{mathrsfs,amsfonts,amsthm,amsmath,amssymb} % thmtools
\usepackage{enumitem}                 % 精确控制列表环境 (如 itemize, enumerate, description) 样式和间距
\usepackage{lineno}                   % to add line numbers for correction

\usepackage[hyperindex, pdfencoding = auto, psdextra, bookmarksdepth = 4]{hyperref}  % 书签与链接的交叉引用
\hypersetup{
    bookmarksopen      = true,
    bookmarksopenlevel = 1,
    bookmarksnumbered  = true,
    pdftitle    = {},
    pdfauthor   = {},
    linktoc     = page,
    anchorcolor = green,
    colorlinks  = true,        % Colours links instead of ugly boxes,
    linkcolor   = black,
    citecolor   = blue,
    filecolor   = blue,
    urlcolor    = blue,
}

\usepackage[capitalise, noabbrev, nameinlink]{cleveref}
\usepackage{thmtools} % 增强和统一定理类环境的定义+管理+引用+修复anchor冲突问题+与cleveref深度集成

\theoremstyle{plain}

\newtheorem{theorem}{Theorem}[section]
\newtheorem{lemma}[theorem]{Lemma}

\newtheorem{proposition}[theorem]{Proposition}
\newtheorem{corollary}[theorem]{Corollary}

\newtheorem{problem}[theorem]{Problem}

\theoremstyle{definition}   % 来自amsthm宏包的定理样式设置命令 \theoremstyle{风格选项:  plain, definition, remark}
\newtheorem{definition}[theorem]{Definition}
\newtheorem{remark}[theorem]{Remark}
\newtheorem{example}[theorem]{Example}

\makeatletter                                         % make之间的部分允许访问内部命令
    \newcommand{\LeftEqNo}{\let\veqno\@@leqno}        % 让公式编号出现在左侧
\makeatother

\numberwithin{equation}{section}                      % 让公式按 section 进行编号

\usepackage[numbers,sort&compress]{natbib}  % 连续编号参考文献

\begin{document}

\

\vspace{-2cm}

\title[Operators with disconnected spectrum in von Neumann algebras]{Operators with disconnected spectrum in von Neumann algebras}

% \author[]{ABC}
% \author[J. Fang]{Junsheng Fang}
% \address{Junsheng Fang, School of Mathematical Sciences, Hebei Normal University, Shijiazhuang, 050024, China}
% \email{jfang@hebtu.edu.cn}
% \thanks{This article was partly supported by Tianyuan Mathematics Research Center. Chunlan Jiang and Junsheng Fang are partly supported by the Hebei Natural Science Foundation (No.A2023205045).}

% \author[C. Jiang]{Chunlan Jiang}
% \address{Chunlan Jiang, School of Mathematical Sciences, Hebei Normal University, Shijiazhuang, 050024, China}
% \email{cljiang@hebtu.edu.cn}

\author{Minghui Ma}
\address{Minghui Ma, School of Mathematical Sciences, Dalian University of Technology, Dalian, 116024, China}
\email{minghuima@dlut.edu.cn}

%\author{Junhao Shen}
%\address{Junhao Shen, Department of Mathematics \& Statistics, University of New Hampshire, Durham, 03824, US}
%\email{junhao.shen@unh.edu}

\author{Rui Shi}
\address{Rui Shi, School of Mathematical Sciences, Dalian University of Technology, Dalian, 116024, China}
\email{ruishi@dlut.edu.cn}
\thanks{Rui Shi, Minghui Ma, and Tianze Wang were supported by NSFC (No.12271074). % and the Liaoning Revitalization Talents Program (No.XLYC2403058). % and the Fundamental Research Funds for the Central Universities (No. DUT23LAB305).
Minghui Ma was also supported by the China Postdoctoral Science Foundation (No.2025M783077) and the Postdoctoral Fellowship Program of CPSF (No.GZC20252022).}

\author{Tianze Wang}
\address{Tianze Wang, School of Mathematical Sciences, Dalian University of Technology, Dalian, 116024, China}
\email{swan0108@mail.dlut.edu.cn}


\begin{abstract}
    Let $\mathcal{M}$ be a von Neumann algebra, $\mathcal{I}$ a weak-operator dense ideal in $\mathcal{M}$, and $\Phi$ a unitarily invariant $\|\cdot\|$-dominating norm on $\mathcal{I}$.
    In this paper, we provide a necessary and sufficient condition on $\Phi$ such that every operator in $\mathcal{M}$ can be expressed as the sum of an operator in $\mathcal{M}$ with disconnected spectrum and an operator in $\mathcal{I}$ whose $\Phi$-norm is arbitrarily small.
    Similarly, if $\mathcal{A}$ is a unital $C^*$-algebra of real rank zero with dimension greater than one and $\mathcal{I}$ is an essential ideal in $\mathcal{A}$, then every element in $\mathcal{A}$ can be written as the sum of an operator in $\mathcal{A}$ with disconnected spectrum and an operator in $\mathcal{I}$ whose norm is arbitrarily small.
\end{abstract}

\maketitle
%\end{frontmatter}

% \vspace{-1cm}

\section{Introduction}

Let $\mathcal{H}$ be a complex Hilbert space and $\mathcal{B}(\mathcal{H})$ the algebra of all bounded operators on $\mathcal{H}$.
An operator $T$ in $\mathcal{B}(\mathcal{H})$ is said to be {\em irreducible} (resp. \emph{strongly irreducible}) if every projection (resp. idempotent) commuting with $T$ is trivial.
Otherwise, $T$ is said to be {\em reducible} (resp. \emph{weakly reducible}).
Since every idempotent is similar to a projection by \cite[Corollary 0.15]{RR73}, an operator $T$ is strongly irreducible if and only if $XTX^{-1}$ is irreducible for every invertible operator $X$.
By the analytic function calculus (see \cite[Proposition 4.11]{Conway-book}), every operator with disconnected spectrum is weakly reducible;
equivalently, the spectrum of every strongly irreducible operator is connected.

In 1968, P. Halmos \cite{Hal68} proved that the set of irreducible operators on a separable complex Hilbert space is a dense $G_{\delta}$ subset in the operator norm topology.
Later, he raised ten problems in Hilbert spaces in  \cite{Hal70}.
Among them, the eighth problem asks whether every operator on a separable infinite-dimensional complex Hilbert space is a  limit of reducible operators.
From a result of D. Herrero and N. Salinas \cite{HS72}, the set of operators on an infinite-dimensional complex Hilbert space with disconnected spectrum is uniformly dense.
As a consequence, the set of weakly reducible operators is also uniformly dense.
The affirmative answer to Halmos' eighth problem is due to D. Voiculescu \cite{Voi76} by applying his remarkable non-commutative Weyl-von Neumann theorem.
It is worth mentioning that D. Herrero and C. Jiang \cite{HJ90} proved that, on a separable infinite-dimensional complex Hilbert space, the set of strongly irreducible operators is uniformly dense in the set of operators with connected spectrum. We refer the reader to \cite{Gil72,JW98-book,JW06-book} for further related results.

Recall that a {\em von Neumann algebra} is a self-adjoint unital subalgebra of $\mathcal{B}(\mathcal{H})$ that is closed in the weak-operator topology.
Moreover, $\mathcal{M}$ is said to be a {\em factor} if its center consists of scalar multiples of the identity.
Factors are classified into type $\mathrm{I}_n$, type $\mathrm{I}_\infty$, type $\mathrm{II}_1$, type $\mathrm{II}_\infty$, and type $\mathrm{III}$ (see \cite[Theorem 6.5.2]{KR2}).
By definition, $\mathcal{B}(\mathcal{H})$ is a type $\mathrm{I}$ factor.
It is natural to generalize the notion of irreducibility of operators in factors.
In a factor $\mathcal{M}$, an operator $T$ is called \emph{reducible} (resp. \emph{weakly reducible}) if there exists a nontrivial projection (resp. idempotent) in $\mathcal{M}$ that commutes with $T$.
Otherwise, $T$ is \emph{irreducible}  (resp. \emph{strongly irreducible})  in $\mathcal{M}$.
Recently, many results related to reducible or irreducible operators have been extended to von Neumann algebras \cite{FJLX-2013,FJS-2023,MSSW25,SS19,SS20,Shi21}.
As in $\mathcal{B}(\mathcal{H})$, every operator with disconnected spectrum in a factor is weakly reducible.
In this paper, we investigate the density of operators with disconnected spectrum in von Neumann algebras (resp. $C^*$-algebras of real rank zero) with perturbations in essential ideals (see \Cref{def essential}).

Throughout, let $\mathcal{M}$ be a von Neumann algebra, $\mathcal{P}(\mathcal{M})$ the set of projections in $\mathcal{M}$, $\mathcal{I}$ a weak-operator dense (two-sided) ideal in $\mathcal{M}$, and $\Phi$ a unitarily invariant norm on $\mathcal{I}$.
For example, $\mathcal{M}$ itself is a weak-operator dense ideal in $\mathcal{M}$ and the operator norm is unitarily invariant on $\mathcal{M}$.
Likewise, the set of finite-rank operators and, more generally, the set of Schatten $p$-class operators for $p\geqslant 1$ form weak-operator dense ideals in $\mathcal{B}(\mathcal{H})$.
Applying unitarily invariant norms on weak-operator dense ideals to operator-theoretic problems is strongly motivated by Voiculescu's work \cite{Voi79}, in which he affirmatively answered Halmos' fourth problem (see \cite{Hal70}) by introducing normed ideals into the study of diagonalization of normal operators.
Meanwhile, the main result of \cite{Voi79} initiated a series of far-reaching developments in von Neumann algebras (see \cite{BSZ18,HMS26,HS18,KW02,LSSW19,LSS20,Sher07,Voi76,Voi79,Voi81,Voi19,Xia24}).
Inspired by this line of research, we establish the following \emph{sup-inf theorem} for unitarily invariant norms.

\begin{theorem}[\textbf{Sup-Inf}]\label{thm main}
    Let $\mathcal{M}$ be a von Neumann algebra with $\dim\mathcal{M}>1$, $\mathcal{I}$ a weak-operator dense ideal in $\mathcal{M}$, and $\Phi$ a unitarily invariant $\|\cdot\|$-dominating norm on $\mathcal{I}$.
    The following statements are equivalent.
    \begin{enumerate}[label= $(\arabic*)$, ref= \arabic*, leftmargin=*] % labelsep = 1 em, start = 1
    \item\label{item mainthm2}  $\sup\limits_{0\ne P\in\mathcal{P}(\mathcal{M})}
        \inf\limits_{\substack{0\ne E\leqslant P \\ E\in\mathcal{I}\cap\mathcal{P}(\mathcal{M})}}\Phi(E)<\infty$.
    \item\label{item mainthm1}  For every $T\in\mathcal{M}$ and $\varepsilon>0$, there exists an operator $X$ in $\mathcal{I}$ with $\Phi(X)<\varepsilon$ such that the spectrum of $T+X$ is disconnected.
    \end{enumerate}
\end{theorem}

We make a few brief comments on \Cref{thm main}, which we hope will contribute to a deeper understanding of the relationship between operator structures and perturbations under unitarily invariant norms.

To gain some immediate intuition for the condition \eqref{item mainthm2} in \Cref{thm main}, one may take $\mathcal{M}$ as a direct sum of infinitely many copies of $\mathcal{B}(\mathcal{H})$.
In this setting, \eqref{item mainthm2} simply requires that
$\sup\{\Phi(E)\colon E~\text{is a minimal projection in}~\mathcal{M}\}<\infty$.

As an application of \Cref{thm main} and \Cref{lem sp-disconnected}, \emph{the set of operators in $\mathcal{M}$ with disconnected spectrum is open and dense in the operator norm topology};
equivalently, \emph{the set of operators in $\mathcal{M}$ with connected spectrum is closed and nowhere dense in the operator norm topology}.
Consequently, in every factor with dimension greater than one, the set of weakly reducible operators is dense in the operator norm (see \cite[Theorem 3.1]{FJS-2023}).
In contrast, the authors of \cite[Theorem 6.7]{SS19} proved that in every non-Gamma type $\mathrm{II}_1$ factor, the set of reducible operators is nowhere dense and not closed in the operator norm topology, while the authors of \cite[Theorem 5.1]{HMS26} proved that in every separable properly infinite factor, the set of reducible operators is uniformly dense.

In contrast with Theorem 1.1, we obtain an analogous result in the setting of $C^*$-algebras. L. Brown and G. Pedersen \cite{BP91} introduced the notion of \emph{real rank zero} for unital $C^*$-algebras.
Including von Neumann algebras as a properly subclass (see \cite[Proposition 1.3]{BP91}), $C^*$-algebras of real rank zero have abundant projections and the structure of this class is close to that of von Neumann algebras (see \cite[Theorem 2.6]{BP91}).
This significantly simplifies the corresponding $K$-theory and classification problems, making real rank zero a cornerstone in the development of the Elliott classification program.
For $C^*$-algebras of real rank zero, we obtain the following result.
See \emph{essential ideals} in unital $C^*$-algebras with a characterization in von Neumann algebras later in \Cref{def essential} and \Cref{lem essential=dense}.

\begin{theorem}\label{thm main-RR0}
    Let $\mathcal{A}$ be a unital $C^*$-algebra of real rank zero with $\dim\mathcal{A}>1$, and $\mathcal{I}$ an essential ideal in $\mathcal{A}$.
    Then for every $T\in\mathcal{A}$ and $\varepsilon >0$, there exists an element $X$ in $\mathcal{I}$ with $\|X\|<\varepsilon$ such that the spectrum of $T+X$ is disconnected.
\end{theorem}

As a direct application of \Cref{thm main-RR0} and \Cref{lem sp-disconnected}, we obtain the following corollary.

\begin{corollary}
    Let $\mathcal{A}$ be a unital $C^*$-algebra of real rank zero with $\dim\mathcal{A}>1$.
    Then the set of elements in $\mathcal{A}$ with disconnected spectrum is open and dense.
\end{corollary}

This leads to the following question.

\begin{problem}\label{problem 1}
    In which unital $C^*$-algebras do elements with disconnected spectrum form a dense subset?
\end{problem}

This paper is organized as follows.
In \Cref{sec preliminary}, we prepare some basic properties of essential ideals, unitarily invariant norms and develop several technical tools from spectral theory.
In \Cref{subsec perturbation-VNA}, we prove \Cref{thm main} based on \Cref{prop main} and \Cref{prop main-converse}.
Then we prove \Cref{thm main-RR0} by applying \Cref{lem small-TE-RR0} and the proof of \Cref{prop main} in \Cref{subsec perturbation-RR0}.
Finally, regarding \Cref{problem 1}, we characterize a necessary and sufficient condition in \Cref{prop subset-R} for a compact subset $X$ of $\mathbb{R}$ such that elements with disconnected spectrum in $C(X)$ form a dense subset.

\section{Preliminaries}\label{sec preliminary}

In this paper, we denote by $\mathcal{M}$ a von Neumann algebra and every ideal is assumed to be \emph{two-sided}.
In von Neumann algebras, we need to deal with ideals that are not closed in the operator norm topology, such as the ideal of Schatten $p$-class operators in $\mathcal{B}(\mathcal{H})$ for every $p\geqslant 1$.
Since the tools we develop for ideals are meant to work for both von Neumann algebras and $C^*$-algebras, we adopt the convention that, unless indicated otherwise, ideals in this paper are not assumed to be closed in the norm topology.

\subsection{Essential ideals and unitarily invariant norms}\label{subsec essential-ideal}

\begin{definition}\label{def Phi}
    Let $\mathcal{M}$ be a von Neumann algebra and $\mathcal{I}$ an ideal in $\mathcal{M}$.
    A norm $\Phi$ on $\mathcal{I}$ is said to be {\em unitarily invariant} if $\Phi(UXV)=\Phi(X)$ for every operator $X\in\mathcal{I}$ and unitary operators $U,V\in\mathcal{M}$.
    Meanwhile, $\Phi$ is called {\em $\|\cdot\|$-dominating} if $\Phi(X)\geqslant \|X\|$ for all $X\in\mathcal{I}$.
\end{definition}

\begin{remark}\label{rem Phi}
    In \cite[Definition 2.1.1]{LSS20}, a norm $\Phi$ on $\mathcal{I}$ is said to be {\em $\|\cdot\|$-dominating} if there exists a positive scalar $\alpha>0$ such that $\Phi(X)\geqslant \alpha \|X\|$ for all $X\in\mathcal{I}$.
    Without loss of generality, we will assume that $\alpha = 1$ in this paper.
\end{remark}

For the reader's convenience, we list some useful properties of unitarily invariant norms from \cite[Lemma 2.1.5]{LSS20}.

\begin{lemma}\label{lem 2.1.5}
    Let $\mathcal{M}$ be a von Neumann algebra, $\mathcal{I}$ an ideal in $\mathcal{M}$, and $\Phi$ a unitarily invariant norm on $\mathcal{I}$.
    \begin{enumerate} [label= $(\arabic*)$, ref= \arabic*, leftmargin=*] % labelsep = 1 em, start = 1
        \item \label{item uin1}  $\Phi(AXB)\leqslant\|A\|\Phi(X)\|B\|$ for all $X\in\mathcal{I}$ and $A,B\in\mathcal{M}$.
        \item \label{item uin2}  For all $X\in\mathcal{I}$, we have $|X|,X^*\in\mathcal{I}$, and $\Phi(X)=\Phi(|X|)=\Phi(X^*)$.
        \item \label{item uin3}  If $X \in\mathcal{M}$, $Y\in\mathcal{I}$, and $0\leqslant X\leqslant Y$, then $X\in\mathcal{I}$ and $\Phi(X)\leqslant\Phi(Y)$.
    \end{enumerate}
\end{lemma}

Let $\mathcal{A}$ be a unital $C^*$-algebra and $\mathcal{A}_{sa}$ the set of self-adjoint operators in $\mathcal{A}$.
According to \cite[Theorem 2.6]{BP91}, we say that $\mathcal{A}$ has \emph{real rank zero} if and only if the set of self-adjoint operators with finite spectrum is dense in $\mathcal{A}_{sa}$.
Note that every von Neumann algebra has real rank zero by \cite[Proposition 1.3]{BP91}.
By \Cref{lem 2.1.5} \eqref{item uin2}, every ideal in a von Neumann algebra is self-adjoint.
The same conclusion is not true for $C^*$-algebras if the ideal is not assumed to be closed.
We include the following example to illustrate this point.

\begin{example}\label{eg non-self-adjoint-ideal}
    Let $\mathcal{A}$ be the $C^*$-algebra of all convergent sequences in $\ell^\infty$, $\theta$ an irrational number, and $X=\{\frac{1}{n}e^{2\pi i n\theta}\}_{n=1}^{\infty}\in\mathcal{A}$.
    Since $\mathcal{A}$ is commutative, $\mathcal{A}X$ is an ideal in $\mathcal{A}$.
    We claim that $\mathcal{A}X$ is not self-adjoint.
    Otherwise, we have $X^*=XY$ for some $Y=\{y_n\}_{n=1}^{\infty}\in\mathcal{A}$.
    It follows that $y_n=e^{-4\pi in\theta}$ and hence $Y$ is not convergent.
    That is a contradiction.
\end{example}

Let $\mathcal{A}$ be a unital $C^*$-algebra and $\mathcal{I}$ an ideal in $\mathcal{A}$.
Then the closure $\overline{\mathcal{I}}$ of $\mathcal{I}$ is a closed ideal in $\mathcal{A}$.
By \cite[Lemma I.5.1]{Davidson-book}, $\overline{\mathcal{I}}$ is self-adjoint.
Thus, $A\mathcal{I}= \{0\}$ if and only if $\mathcal{I}A^*= \{0\}$ for every operator $A\in\mathcal{A}$.
In the following definition, essential ideals are not assumed to be closed in the operator norm topology.

\begin{definition}\label{def essential}
    Let $\mathcal{A}$ be a unital $C^*$-algebra.
    An ideal $\mathcal{I}$ of $\mathcal{A}$ is said to be \emph{essential} if $A\mathcal{I}= \{0\}$ implies that $A=0$ for every operator $A\in\mathcal{A}$; equivalently, $\mathcal{I}A= \{0\}$ implies that $A=0$ for every operator $A\in\mathcal{A}$.
\end{definition}

In contrast with \cite[Definition 3.12.7]{Ped79}, we present an auxiliary lemma for the reader's convenience.
For subsets $\mathcal{S}_1$ and $\mathcal{S}_2$ in $\mathcal{A}$, denote by $\mathcal{S}_1\mathcal{S}_2$ the set of all products $A_1A_2$ for $A_j \in \mathcal{S}_j$ and $j=1,2$.

\begin{lemma}
    Let $\mathcal{A}$ be a unital $C^*$-algebra and $\mathcal{I}$ an ideal in $\mathcal{A}$.
    Then the following statements are equivalent.
    \begin{enumerate} [label= $(\arabic*)$, ref= \arabic*, leftmargin=*] % labelsep = 1 em, start = 1
        \item \label{item ess-ideal1}  $\mathcal{I}$ is an essential ideal in $\mathcal{A}$.
        \item \label{item ess-ideal2}  $\mathcal{I} \cap \mathcal{J} \ne \{0\}$ for every nonzero ideal $\mathcal{J}$ in $\mathcal A$.
    \end{enumerate}
\end{lemma}

\begin{proof}
    $\eqref{item ess-ideal1}\Rightarrow \eqref{item ess-ideal2}$.
    Let $\mathcal{J} \subseteq \mathcal{A}$ be a nonzero ideal and choose $A \in \mathcal{J}$ with $A\ne 0$.
    Since $\mathcal{I}$ is essential, we have $\mathcal{I}A \ne \{0\}$, and hence $\mathcal{IJ} \ne \{0\}$.
    This together with $\mathcal{IJ} \subseteq \mathcal{I}\cap \mathcal{J}$ yields that $\mathcal{I}\cap \mathcal{J} \ne \{0\}$.

    $\eqref{item ess-ideal2}\Rightarrow \eqref{item ess-ideal1}$.
    Let $A\in \mathcal{A}$ with $A\ne 0$ and  $\mathcal{J} =\mathrm{span}\mathcal{A} A^* \mathcal{A}$.
    Then $\mathcal{J}$ is a nonzero ideal.
    From $\eqref{item ess-ideal2}$, there is a nonzero element $B$ in $\mathcal{I} \cap \mathcal{J} $ of the form $B = \sum_j X_j A^* Y_j$.
    Since $BB^*\ne 0$, we have $\mathcal{I}B^*\ne \{0\}$, and hence $\mathcal{I}A\ne \{0\}$.
    Thus, $\mathcal{I}$ is essential.
\end{proof}

The following lemma is prepared for \Cref{def f-Phi}.

\begin{lemma}\label{lem E<P-RR0}
    Let $\mathcal{A}$ be a unital $C^*$-algebra of real rank zero and $\mathcal{I}$ an essential ideal in $\mathcal{A}$.
    Then for every nonzero projection $P$ in $\mathcal{A}$, there exists a nonzero projection $E$ in $\mathcal{I}$ such that $E\leqslant P$.
\end{lemma}

\begin{proof}
    Since $\mathcal{I}$ is essential, there exists an operator $T$ in $\mathcal{I}$ such that $PT\ne 0$.
    Since $\mathcal{A}$ has real rank zero, there are nonzero projections $\{P_j\}_{j=1}^n$ in $\mathcal{A}$ with sum $I$ and scalars $0\leqslant\lambda_1\leqslant\cdots\leqslant\lambda_n=\|PT\|^2$ such that
    \begin{equation*}
        \|T^*PT-(\lambda_1P_1+\cdots+\lambda_nP_n)\|<\frac{1}{2}\lambda_n.
    \end{equation*}
    It follows that $\|P_nT^*PTP_n-\lambda_nP_n\|<\frac{1}{2}\lambda_n$.
    Hence $P_nT^*PTP_n$ is nonzero and
    \begin{equation*}
        \sigma(P_nT^*PTP_n)\subseteq\{0\}\cup
        \left[\frac{1}{2}\lambda_n,\frac{3}{2}\lambda_n\right].
    \end{equation*}
    By the well-known equality $\sigma(AB)\setminus\{0\}=\sigma(BA)\setminus\{0\}$ for all operators $A$ and $B$, we obtain that
    \begin{equation*}
        \sigma(PTP_nT^*P)\subseteq\{0\}\cup
        \left[\frac{1}{2}\lambda_n,\frac{3}{2}\lambda_n\right].
    \end{equation*}
    Moreover, $A=PTP_nT^*P\in\mathcal{I}$ is nonzero.
    Let $\chi$ be the characteristic function of the closed interval $[\frac{1}{2}\lambda_n,\frac{3}{2}\lambda_n]$, $f(t)=t^{-1}\chi(t)$, and $E=\chi(A)=Af(A)$.
    Then $E$ is a nonzero projection in $\mathcal{I}$ with $E\leqslant P$.
    This completes the proof.
\end{proof}

By the following consequence of \Cref{lem E<P-RR0}, we will use the terminology {\em weak-operator dense ideals} instead of {\em essential ideals} in von Neumann algebras.

\begin{corollary}\label{lem essential=dense}
    Let $\mathcal{M}$ be a von Neumann algebra.
    Then an ideal in $\mathcal{M}$ is essential if and only if it is weak-operator dense.
\end{corollary}

\begin{proof}
    It is clear that every weak-operator dense ideal is essential.
    Suppose that $\mathcal{I}$ is an essential ideal in $\mathcal{M}$.
    Let $\{E_a\}_{a\in \mathbb{A}}$ be a maximal orthogonal family of nonzero projections in $\mathcal{I}$.
    Let $P = I-\sum_{a\in \mathbb{A}}E_a$.
    If $P\ne 0$, then there exists a nonzero projection $E_0$ in $\mathcal{I}$ such that $E_0\leqslant P$ by \Cref{lem E<P-RR0}.
    That contradicts the maximality of $\{E_a\}_{a\in \mathbb{A}}$.
    It follows that $I=\sum_{a\in \mathbb{A}}E_a$.
    We put $E_\Lambda=\sum_{a\in\Lambda}E_a\in\mathcal{I}$ for every finite subset $\Lambda\subseteq \mathbb{A}$.
    Then for every $T$ in $\mathcal{M}$, $TE_\Lambda$ is weak-operator convergent to $T$ as $\Lambda\to \mathbb{A}$.
    This completes the proof.
\end{proof}

For each unitarily invariant norm $\Phi$ on a weak-operator dense ideal in a von Neumann algebra, we define a constant $c_\Phi$ associated with $\Phi$ as follows.

\begin{definition}\label{def f-Phi}
    Let $\mathcal{M}$ be a von Neumann algebra, $\mathcal{I}$ a weak-operator dense ideal in $\mathcal{M}$, and $\Phi$ a unitarily invariant norm on $\mathcal{I}$.
    For every nonzero projection $P$ in $\mathcal{M}$, we put
    \begin{equation*}
        f_\Phi(P)=\inf\{\Phi(E)\colon E\in\mathcal{I}~\text{is a nonzero subprojection of}~P\}.
    \end{equation*}
    The constant $c_\Phi$ is defined as
    \begin{equation*}
        c_\Phi = \sup\{f_\Phi(P)\colon P~\text{is a nonzero projection in}~\mathcal{M}\}.
    \end{equation*}
\end{definition}

\begin{remark}\label{rem f-Phi}
    Note that $0\leqslant f_\Phi(P)<\infty$ for every nonzero projection $P$ in $\mathcal{M}$.
    By definition, $f_\Phi$ is decreasing, i.e., $f_\Phi(P)\geqslant f_\Phi(Q)$ for all nonzero projections $P$ and $Q$ in $\mathcal{M}$ with $P\leqslant Q$.
    Moreover, if $c_\Phi<\infty$, $\delta>0$, and $P$ is a nonzero projection in $\mathcal{M}$, then there exists a nonzero projection $E$ in $\mathcal{I}$ such that $E\leqslant P$ and $\Phi(E)<c_\Phi+\delta$. This fact will be used tacitly in what follows.
\end{remark}

For every projection $P$ in $\mathcal{M}$, its {\em central support} is denoted by $C_P$.
By the following lemma, the value of $f_\Phi$ at every nonzero projection $P$ in $\mathcal{M}$ only depends on $C_P$, i.e., $f_\Phi(P)=f_\Phi(Q)$ for all nonzero projections $P$ and $Q$ in $\mathcal{M}$ with $C_P=C_Q$.

\begin{lemma}\label{lem carrier}
    Let $\mathcal{M}$ be a von Neumann algebra, $\mathcal{I}$ a weak-operator dense ideal in $\mathcal{M}$, and $\Phi$ a unitarily invariant norm on $\mathcal{I}$.
    Then $f_\Phi(P)=f_\Phi(C_P)$ for every nonzero projection $P$ in $\mathcal{M}$.
\end{lemma}

\begin{proof}
    It is clear that $f_\Phi(P)\geqslant f_\Phi(C_P)$ as $f_\Phi$ is decreasing.
    For any $\varepsilon>0$, there exists a nonzero projection $E$ in $\mathcal{I}$ such that $E\leqslant C_P$ and $\Phi(E)<f_\Phi(C_P)+\varepsilon$.
    Since $C_EC_P=C_E\ne 0$, there exists a nonzero partial isometry $V$ in $\mathcal{M}$ such that $V^*V\leqslant E$ and $VV^*\leqslant P$ by \cite[Proposition 6.1.8]{KR2}.
    From \Cref{lem 2.1.5} \eqref{item uin1}, we have
    \begin{equation*}
        \Phi(VV^*)=\Phi(VEV^*)\leqslant\Phi(E)<f_\Phi(C_P)+\varepsilon.
    \end{equation*}
    It follows that $f_\Phi(P)<f_\Phi(C_P)+\varepsilon$.
    This completes the proof.
\end{proof}

The following two lemmas are prepared for \Cref{prop main-converse}.

\begin{lemma}\label{lem c=infinite}
    Let $\mathcal{M}$ be a von Neumann algebra, $\mathcal{I}$ a weak-operator dense ideal in $\mathcal{M}$, and $\Phi$ a unitarily invariant norm on $\mathcal{I}$ with $c_\Phi=\infty$.
    Then for every integer $n\geqslant 1$, there exists a central projection $Z$ in $\mathcal{M}$ with $0<Z<I$ such that $f_\Phi(Z)\geqslant n$ and
    \begin{equation*}
        \sup\{f_\Phi(P)\colon P~\text{is a nonzero central projection in}~\mathcal{M}(I-Z)\}=\infty.
    \end{equation*}
\end{lemma}

\begin{proof}
    By \Cref{lem carrier}, there exists a nonzero central projection $Z_0$ in $\mathcal{M}$ such that $f_\Phi(Z_0)\geqslant n$.
    Note that
    \begin{equation*}
        \mathcal{M}=\mathcal{M}Z_0\oplus\mathcal{M}(I-Z_0)\quad\text{and}\quad
        \mathcal{I}=\mathcal{I}Z_0\oplus\mathcal{I}(I-Z_0).
    \end{equation*}
    Let $\Phi_0=\Phi|_{\mathcal{I}Z_0}$ and $\Phi_0'=\Phi|_{\mathcal{I}(I-Z_0)}$ be restrictions of $\Phi$.
    Then $c_{\Phi_0}=\infty$ or $c_{\Phi_0'}=\infty$.
    If $c_{\Phi_0'}=\infty$, then $Z_0\ne I$ and we can take $Z=Z_0$.
    Suppose that $c_{\Phi_0}=\infty$.
    Then there exists a central projection $Z_1$ in $\mathcal{M}$ such that $0<Z_1<Z_0$.
    Similarly, we have
    \begin{equation*}
        \mathcal{M}Z_0=\mathcal{M}Z_1\oplus\mathcal{M}(Z_0-Z_1)\quad\text{and}\quad
        \mathcal{I}Z_0=\mathcal{I}Z_1\oplus\mathcal{I}(Z_0-Z_1).
    \end{equation*}
    Let $\Phi_1=\Phi|_{\mathcal{I}Z_1}$ and $\Phi_1'=\Phi|_{\mathcal{I}(Z_0-Z_1)}$.
    Then $c_{\Phi_1}=\infty$ or $c_{\Phi_1'}=\infty$.
    If $c_{\Phi_1'}=\infty$, then we take $Z=Z_1\leqslant Z_0$.
    If $c_{\Phi_1}=\infty$, then we take $Z=Z_0-Z_1\leqslant Z_0$.
    In both cases, we have $f_\Phi(Z)\geqslant f_\Phi(Z_0)\geqslant n$.
    This completes the proof.
\end{proof}

\begin{lemma}\label{lem norm<Phi}
    Let $\mathcal{M}$ be a von Neumann algebra, $\mathcal{I}$ a weak-operator dense ideal in $\mathcal{M}$, and $\Phi$ a unitarily invariant norm on $\mathcal{I}$.
    % Then for any operator $X$ in $\mathcal{I}$ and nonzero central projection $Z$ in $\mathcal{M}$, we have
    Then the inequality
    \begin{equation*}
        \|XZ\|f_\Phi(Z)\leqslant\Phi(X)
    \end{equation*}
    holds for every operator $X$ in $\mathcal{I}$ and nonzero central projection $Z$ in $\mathcal{M}$.
\end{lemma}

\begin{proof}
    By \Cref{lem 2.1.5} \eqref{item uin2}, we may assume that $X$ is a positive operator in $\mathcal{I}$.
    Without loss of generality, we assume that $XZ\ne 0$.
    Then for any $0<a<\|XZ\|$, there exists a nonzero projection $E$ in $\mathcal{M}Z$ such that $EXE\geqslant aE$.
    It follows from \Cref{lem 2.1.5} that $E\in\mathcal{I}$ and
    \begin{equation*}
        af_\Phi(Z)\leqslant a\Phi(E)=\Phi(aE)\leqslant\Phi(EXE)\leqslant\Phi(X).
    \end{equation*}
    We complete the proof by the arbitrariness of $a$.
\end{proof}

\begin{remark}\label{rem norm<Phi}
    As a consequence, $\|X\|f_\Phi(I)\leqslant\Phi(X)$ for all $X\in\mathcal{I}$.
    Therefore, the unitarily invariant norm $\Phi$ on $\mathcal{I}$ is $\|\cdot\|$-dominating if and only if $f_\Phi(I)\geqslant 1$.
\end{remark}

\subsection{Tools from spectral theory}\label{subsec spectral-theory}

Suppose that $\mathcal{A}$ and $\mathcal{B}$ are unital $C^*$-algebras satisfying that $I_{\mathcal{B}}\in \mathcal{A} \subseteq \mathcal{B}$.
Then $\sigma_{\mathcal{A}}(A) = \sigma_{\mathcal{B}}(A)$ for every operator $A$ in $\mathcal{A}$ by \cite[Proposition VIII.1.14]{Conway-book}.
Moreover, by the GNS construction \cite[Theorem I.9.12]{Davidson-book}, for each abstract $C^*$-algebra $\mathcal{A}$, there is a faithful *-representation $\varphi\colon\mathcal{A}\to\mathcal{B}(\mathcal{H})$ and $\mathcal{A}$ is \emph{isometrically *-isomorphic} to the concrete $C^*$-algebra $\varphi(\mathcal{A})$.
It follows that $\sigma(A)=\sigma(\varphi(A))$ for every operator $A$ in $\mathcal{A}$ and the study of the spectrum of $A$ in $\mathcal{A}$ can always be reduced to that of $\varphi(A)$ in $\mathcal{B(H)}$.
Recall that the {\em approximate point spectrum} of an operator $T$ acting on $\mathcal{H}$ is denoted by $\sigma_{ap}(T)$.

\begin{remark}\label{rem sigma-ap}
    Let $\mathcal{A}$ be a unital $C^*$-algebra, $T$ an operator in $\mathcal{A}$, and $\lambda$ a scalar in the boundary $\partial\sigma(T)$ of $\sigma(T)$.
    Then $(T-\lambda I)^*(T-\lambda I)$ is not invertible in $\mathcal{A}$.
    Indeed, we may assume that $\mathcal{A}$ is a concrete $C^*$-algebra, i.e., it acts on a complex Hilbert space $\mathcal{H}$.
    Since $\partial\sigma(T)\subseteq\sigma_{ap}(T)$, there exists a sequence of unit vectors $\{x_n\}_{n=1}^{\infty}$ in $\mathcal{H}$ such that $\|(T-\lambda I)x_n\|\to 0$ as $n\to\infty$.
    Then $\|(T-\lambda I)^*(T-\lambda I)x_n\|\to 0$.
    This completes the proof.
\end{remark}

In the remainder of this section, we develop some technical tools for operators in upper-triangular form arising from spectral theory, which are also of independent interest.

\begin{lemma}\label{lem UT-spectrum}
    Let $\mathcal{H}_1$ and $\mathcal{H}_2$ be complex Hilbert spaces and $\mathcal{H}=\mathcal{H}_1\oplus\mathcal{H}_2$.
    Suppose that $T$ is an operator in $\mathcal{B}(\mathcal{H})$ of the form
    \begin{equation}\label{equ UT}
        T=
        \begin{pmatrix}
            T_1 & T_{12} \\
            0 & T_2
        \end{pmatrix}
        \begin{matrix}
            \mathcal{H}_1 \\
            \mathcal{H}_2
        \end{matrix}.
    \end{equation}
    Then the spectrum of $T$ satisfies that
    \begin{equation*}
        \sigma_{ap}(T_1)\cup\big(\sigma(T_2)\setminus\sigma(T_1)\big)\subseteq\sigma(T)
        \subseteq\sigma(T_1)\cup\sigma(T_2).
    \end{equation*}
\end{lemma}

\begin{proof}
    If both $T_1$ and $T_2$ are invertible, then $T$ is an invertible operator whose inverse is given by
    \begin{equation*}
        T^{-1}=
        \begin{pmatrix}
            T_1^{-1} & -T_1^{-1}T_{12}T_2^{-1} \\
            0 & T_2^{-1}
        \end{pmatrix}
        \begin{matrix}
            \mathcal{H}_1 \\
            \mathcal{H}_2
        \end{matrix}.
    \end{equation*}
    It follows that $\sigma(T)\subseteq\sigma(T_1)\cup\sigma(T_2)$.

    For every $\lambda\in\sigma_{ap}(T_1)$, there exists a sequence of unit vectors $\{x_n\}_{n=1}^{\infty}$ in $\mathcal{H}_1$ such that $\|(T_1-\lambda I_{\mathcal{H}_1})x_n\|\to 0$ as $n\to\infty$.
    Let $y_n=x_n\oplus 0\in\mathcal{H}$.
    Then $\|(T-\lambda I)y_n\|\to 0$ and hence $\lambda\in\sigma_{ap}(T)$.
    It follows that $\sigma_{ap}(T_1)\subseteq\sigma_{ap}(T)\subseteq\sigma(T)$.

    Suppose that $T_1$ is invertible.
    Clearly, $T$ is not surjective whenever $T_2$ is not surjective.
    Moreover, if $x_2\in\ker T_2$, then $(-T_1^{-1}T_{12}x_2)\oplus x_2\in\ker T$.
    Hence $T$ is not injective whenever $T_2$ is not injective.
    It follows that $\sigma(T_2)\setminus\sigma(T_1)\subseteq\sigma(T)$.
    % This completes the proof.
\end{proof}

By the equality $\sigma(T_1)\cup\sigma(T_2)=\sigma(T_1)\cup\big(\sigma(T_2)\setminus\sigma(T_1)\big)$, the following corollary is a direct consequence of \Cref{lem UT-spectrum}.

\begin{corollary}\label{cor UT-spectrum}
    Let $T$ be the operator given by \eqref{equ UT}.
    If $\sigma(T_1)=\sigma_{ap}(T_1)$, then $\sigma(T)=\sigma(T_1)\cup\sigma(T_2)$.
    In particular, if $T_1$ is normal, then $\sigma(T)=\sigma(T_1)\cup\sigma(T_2)$.
\end{corollary}

The following lemma is a certain converse of \Cref{cor UT-spectrum}.

\begin{lemma}\label{lem UT-converse}
    Let $T_1$ be an operator on a Hilbert space $\mathcal{H}_1$.
    If $\sigma(T_1)\ne\sigma_{ap}(T_1)$, then there exists an operator $T$ of the form \eqref{equ UT} such that $\sigma(T)\ne\sigma(T_1)\cup\sigma(T_2)$.
\end{lemma}

\begin{proof}
    Without loss of generality, we assume that $0\in\sigma(T_1)\setminus\sigma_{ap}(T_1)$.
    Then $T_1$ is injective and has closed range in $\mathcal{H}_1$.
    Let $\mathcal{H}_2$ be the direct sum of countably infinitely many copies of $\mathcal{H}_1$ and $\mathcal{H}=\mathcal{H}_1\oplus\mathcal{H}_2$.
    Then there exists a partial isometry $V$ in $\mathcal{B}(\mathcal{H})$ with initial space $\mathcal{H}_2$ and final space $\mathcal{H}\ominus\mathrm{ran}(T_1) $.
    Let
    \begin{equation*}
        T_{12}=I_{\mathcal{H}_1}VI_{\mathcal{H}_2}\in\mathcal{B}(\mathcal{H}_2,\mathcal{H}_1)
        \quad\text{and}\quad
        T_2=I_{\mathcal{H}_2}VI_{\mathcal{H}_2}\in\mathcal{B}(\mathcal{H}_2).
    \end{equation*}
    Then the operator $T$ given by \eqref{equ UT} is invertible in $\mathcal{B}(\mathcal{H})$.
    In other words, $0\notin\sigma(T)$.
    It follows that $\sigma(T)$ is a proper subset of $\sigma(T_1)\cup\sigma(T_2)$.
\end{proof}

\begin{remark}\label{rem UT-converse}
    Let $T_1$ be an operator on a Hilbert space $\mathcal{H}_1$.
    Then there exists an operator $T$ of the form \eqref{equ UT} such that $\sigma(T)=\sigma_{ap}(T_1)$ by \cite{Read87}.
    This result is stronger than \Cref{lem UT-converse}, and its proof relies on a substantially more involved construction than that of \Cref{lem UT-converse}.
    % which is stronger than \Cref{lem UT-converse} and its proof employs deeper construction than that of \Cref{lem UT-converse}.
    In addition, if $T$ is of the form \eqref{equ UT}, then we can write
    \begin{equation*}
        T^*=
        \begin{pmatrix}
            T_2^* & T_{12}^* \\
            0 & T_1^*
        \end{pmatrix}
        \begin{matrix}
            \mathcal{H}_2 \\
            \mathcal{H}_1
        \end{matrix}.
    \end{equation*}
    It follows from \Cref{cor UT-spectrum} that $\sigma(T)=\sigma(T_1)\cup\sigma(T_2)$ whenever $\sigma(T_2^*)=\sigma_{ap}(T_2^*)$.
\end{remark}

We present an example to illustrate \Cref{lem UT-converse}.

\begin{example}\label{eg shift}
    Suppose that $\mathcal{H}$ is a complex Hilbert space with an orthonormal basis $\{e_n\}_{n\in\mathbb{Z}}$ and $T$ is the bilateral shift operator on $\mathcal{H}$ satisfying that $Te_n = e_{n+1}$ for every integer $n\in\mathbb{Z}$.
    Let $\mathcal{H}_1$ and $\mathcal{H}_2$ be the closed linear spans of $\{e_n\}_{n\geqslant 0} $ and $\{e_n\}_{n \leqslant -1}$, respectively.
    Since $\mathcal{H}_1$ is an invariant subspace of $T$, we can write
    \begin{equation*}
        T=
        \begin{pmatrix}
            T_1 & T_{12} \\
            0 & T_2
        \end{pmatrix}
        \begin{matrix}
            \mathcal{H}_1 \\
            \mathcal{H}_2
        \end{matrix}
    \end{equation*}
    as in \eqref{equ UT}, where $T_{12}$ is a rank-one partial isometry, and both $T_1$ and $T_2^*$ are unitarily equivalent to the unilateral shift operator.
    Therefore, we have
    \begin{equation*}
        \sigma(T)=\sigma_{ap}(T_1)=\sigma_{ap}(T_2^*)=\mathbb{T} \quad\text{and}\quad \sigma(T_1)=\sigma(T_2)=\overline{\mathbb{D}},
    \end{equation*}
    where $\mathbb{T}$ is the unit circle and $\overline{\mathbb{D}}$ is the closed unit disk in the complex plane $\mathbb{C}$.
    It follows that $\sigma(T)$ is a proper subset of $\sigma(T_1)\cup\sigma(T_2)$.
\end{example}

At the end of this section, we present a standard tool in spectral theory.
For the reader's convenience, we sketch a brief proof.

\begin{lemma}\label{lem sp-disconnected}
    Let $\mathcal{A}$ be a unital Banach algebra.
    Then the set of elements in $\mathcal{A}$ with disconnected spectrum is open in the norm topology.
\end{lemma}

\begin{proof}
    Let $T$ be an element in $\mathcal{A}$ with disconnected spectrum.
    Then there are disjoint bounded open subset $\Omega_1$ and $\Omega_2$ of $\mathbb{C}$ such that
    \begin{equation*}
        \sigma(T)\cap\Omega_1\ne\varnothing, \quad\sigma(T)\cap\Omega_2\ne\varnothing, \quad\text{and}\quad
        \sigma(T)\subseteq\Omega_1\cup\Omega_2.
    \end{equation*}
    Let $\chi_j$ be the characteristic function of $\Omega_j$ for $j=1,2$.
    Then $\chi_1(T)$ and $\chi_2(T)$ are nonzero idempotents in $\mathcal{A}$ with $\chi_1(T)+\chi_2(T)=I$.

    Arguing by contradiction, if $T$ is not an interior point of the set of elements with disconnected spectrum, then there exists a sequence $\{T_n\}_{n=1}^{\infty}$ with connected spectrum converging to $T$.
    By the upper semi-continuity of the spectrum \cite[Theorem 10.20]{Rudin-book-FA}, there exists $n_1\geqslant 1$ such that $\sigma(T_n)\subseteq\Omega_1\cup\Omega_2$ for each $n\geqslant n_1$.
    By the continuity of the analytic function calculus, there exists $n_2\geqslant n_1$ such that
    \begin{equation*}
        \|\chi_j(T_n)-\chi_j(T)\|<\|\chi_j(T)\|
    \end{equation*}
    for each $j=1,2$ and each $n\geqslant n_2$.
    It follows that $\chi_j(T_n)\ne 0$ for $j=1,2$.
    Thus, the spectrum of $T_n$ is disconnected for all $n\geqslant n_2$, a contradiction.
    This completes the proof.
\end{proof}

\section{Main results} \label{sec main}
In this section, we first present the proofs of \Cref{prop main} and \Cref{prop main-converse}.
Our main result, \Cref{thm main}, follows directly from these two propositions. Then  we prove \Cref{thm main-RR0} with \Cref{lem small-TE-RR0}. Furthermore, in \Cref{prop subset-R}, a necessary and sufficient condition  is given for a compact subset $X$ of $\mathbb{R}$ such that elements of $C(X)$ with disconnected spectrum  form a dense subset, in connection with \Cref{problem 1}.

\subsection{Perturbations in von Neumann algebras}\label{subsec perturbation-VNA}
The following technical lemma is prepared for \Cref{prop main}.

\begin{lemma}\label{lem small-TE}
    Let $\mathcal{M}$ be a von Neumann algebra, $\mathcal{I}$ a weak-operator dense ideal in $\mathcal{M}$, and $\Phi$ a unitarily invariant $\|\cdot\|$-dominating norm on $\mathcal{I}$ with $c_\Phi<\infty$.
    Suppose that $T$ is an operator in $\mathcal{M}$ and $\lambda$ is a boundary point of $\sigma(T)$, i.e., $\lambda\in\partial\sigma(T)$.
    Then for any $\varepsilon>0$, there exists a nonzero projection $E$ in $\mathcal{I}$ such that
    \begin{equation*}
        \Phi(E)<1+c_\Phi\quad\text{and}\quad \Phi((T-\lambda I)E)<\varepsilon.
    \end{equation*}
\end{lemma}

\begin{proof}
    Without loss of generality, we may assume that $\lambda=0\in\partial\sigma(T)$.
    Then $T^*T$ is not invertible by \Cref{rem sigma-ap}.
    Let $P$ be the spectral projection of $T^*T$ with respect to the closed interval $[0,\frac{\varepsilon^2}{(1+c_\Phi)^2}]$.
    It is clear that $P$ is a nonzero projection in $\mathcal{M}$ with $\|TP\|\leqslant\frac{\varepsilon}{1+c_\Phi}$.
    By \Cref{lem E<P-RR0}, there exists a nonzero projection $E_0$ in $\mathcal{I}$ such that $E_0\leqslant P$.
    By \Cref{rem f-Phi}, there exists a nonzero projection $E\leqslant E_0$ such that $\Phi(E)<1+c_\Phi$.
    Since $E$ is a subprojection of $P$, we have $\|TE\|\leqslant\|TP\|\leqslant\frac{\varepsilon}{1+c_\Phi}$.
    Combining with \Cref{lem 2.1.5} \eqref{item uin1}, we have
    \begin{equation*}
      \Phi(TE)\leqslant\|TE\|\Phi(E)<\varepsilon.
    \end{equation*}
    This completes the proof.
\end{proof}

\begin{remark}\label{rem small-TE}
    Every nonzero subprojection of $E$ also satisfies the conclusion in \Cref{lem small-TE}.
    In particular, if the von Neumann algebra $\mathcal{M}$ is atomic, i.e., a direct sum of type I factors, then $E$ can be chosen to be a minimal projection.
    If $\mathcal{M}$ is a type $\mathrm{II}_\infty$ factor with a normal faithful semifinite tracial weight $\tau$, then we can require that $\tau(E)<\varepsilon$.
\end{remark}

The following proposition is the main result in this section.

\begin{proposition}\label{prop main}
    Let $\mathcal{M}$ be a von Neumann algebra with $\dim\mathcal{M}>1$, $\mathcal{I}$ a weak-operator dense ideal in $\mathcal{M}$, and $\Phi$ a unitarily invariant $\|\cdot\|$-dominating norm on $\mathcal{I}$ with $c_\Phi<\infty$.
    Then for any $T\in\mathcal{M}$ and $\varepsilon>0$, there exists an operator $X$ in $\mathcal{I}$ with $\Phi(X)<\varepsilon$ such that the spectrum of $T+X$ is disconnected.
\end{proposition}

\begin{proof}
    Let $\varepsilon_0=\frac{\varepsilon}{2(1+c_\Phi)}>0$.
    Suppose that $\lambda$ is a scalar in $\sigma(T)$ such that
    \begin{equation*}
        \operatorname{Re}\lambda=\max\{\operatorname{Re}z\colon z\in\sigma(T)\}.
    \end{equation*}
    Then $\lambda\in\partial\sigma(T)$.
    Without loss of generality, we may assume that $\lambda=0$.
    By the upper semi-continuity of the spectrum \cite[Theorem 10.20]{Rudin-book-FA}, there exists $0<\delta<\varepsilon_0$ such that for every operator $A$ with $\|A-T \|<\delta$ and every $z\in\sigma(A)$, we have $\operatorname{Re}z<\varepsilon_0$.
    By \Cref{lem small-TE}, there exists a nonzero projection $E$ in $\mathcal{I}$ such that
    \begin{equation}\label{equ main}
        \Phi(E)<1+c_\Phi\quad\text{and}\quad \Phi(TE)<\delta.
    \end{equation}
    Since $\dim\mathcal{M}>1$, we may further assume that $E\ne I$.
    This yields that $\operatorname{Re}z<\varepsilon_0$ for every $z\in\sigma(T(I-E))$ as $\|T(I-E)-T\|=\|TE\|\leqslant\Phi(TE)<\delta$.
    It follows from \Cref{cor UT-spectrum} that
    \begin{equation*}
        \sigma(T(I-E))=\{0\}\cup\sigma_{(I-E)\mathcal{M}(I-E)}((I-E)T(I-E)).
    \end{equation*}
    As a consequence, $\operatorname{Re}z<\varepsilon_0$ for every $z\in\sigma_{(I-E)\mathcal{M}(I-E)}((I-E)T(I-E))$.
    By applying \Cref{cor UT-spectrum} again, the set
    \begin{equation*}
        \sigma(\varepsilon_0E+T(I-E))
        =\{\varepsilon_0\}\cup\sigma_{(I-E)\mathcal{M}(I-E)}((I-E)T(I-E))
    \end{equation*}
    is disconnected.
    Let $X=(\varepsilon_0I-T)E$.
    Then, by \eqref{equ main} and the definition of $\varepsilon_0$, we obtain that
    \begin{equation*}
        \Phi(X)\leqslant\varepsilon_0\Phi(E)+\Phi(TE)<\varepsilon.
    \end{equation*}
    This completes the proof.
\end{proof}

\begin{remark}\label{rem main}
    By \Cref{rem small-TE} and the proof of \Cref{prop main}, if $\mathcal{M}$ is atomic, then $X$ can be taken as a rank-one operator.
    Therefore, we get a new proof of the main theorem in \cite{HS72} for Hilbert spaces.
    Moreover, if $\mathcal{M}$ is a type $\mathrm{II}_\infty$ factor with a normal faithful semifinite tracial weight $\tau$, then we can require that $\tau(R(X))<\varepsilon$.
\end{remark}

% \begin{corollary}\label{cor main}
%     Let $\mathcal{M}$ be a von Neumann algebra with $\dim\mathcal{M}>1$.
%     Then the set of operators in $\mathcal{M}$ with disconnected spectrum is dense in $\mathcal{M}$ with respect to the operator norm topology.
% \end{corollary}

The following proposition is a converse of \Cref{prop main}.

\begin{proposition}\label{prop main-converse}
    Let $\mathcal{M}$ be a von Neumann algebra, $\mathcal{I}$ a weak-operator dense ideal in $\mathcal{M}$, and $\Phi$ a unitarily invariant $\|\cdot\|$-dominating norm on $\mathcal{I}$ with $c_\Phi =\infty$.
    Then there exists an operator $T$ in $\mathcal{M}$ such that $\sigma(T+X)$ is connected for every $X$ in $\mathcal{I}$ with $\Phi(X)<1$.
\end{proposition}

\begin{proof}
    By \Cref{lem c=infinite}, there is an orthogonal sequence $\{Z_n\}_{n=1}^{\infty}$ of nonzero central projections in $\mathcal{M}$ such that $f_\Phi(Z_n)\geqslant 2^n$ for each $n\geqslant 1$.
    We put $Z_0=I-\sum_{n=1}^{\infty}Z_n$.
    Let $\{\lambda_n\}_{n=0}^{\infty}$ be a dense sequence in the unit closed disk $\overline{\mathbb{D}}\subseteq\mathbb{C}$ such that $|\lambda_n|\leqslant 1-2^{-n}$ for each $n\geqslant 0$.
    Define an operator $T$ in $\mathcal{M}$ by
    \begin{equation*}
        T=\sum_{n=0}^{\infty}\lambda_nZ_n.
    \end{equation*}
    Let $X$ be an operator in $\mathcal{I}$ with $\Phi(X)<1$.
    By \Cref{lem norm<Phi}, we have $\|XZ_n\|\leqslant 2^{-n}$ for each $n\geqslant 0$.
    It follows that
    \begin{equation*}
        \|T+X\|=\sup_{n\geqslant 0}\|\lambda_nZ_n+XZ_n\|\leqslant 1.
    \end{equation*}
    Hence $\sigma(T+X)\subseteq\overline{\mathbb{D}}$.
    For every scalar $\lambda\in\overline{\mathbb{D}}$, there exists a subsequence $\{\lambda_{n_j}\}_{j=0}^{\infty}$ that is convergent to $\lambda$.
    For each $j\geqslant 0$, choose
    \begin{equation*}
        \mu_j\in\sigma_{\mathcal{M}Z_{n_j}}(\lambda_{n_j}Z_{n_j}+XZ_{n_j})
        \subseteq\sigma(T+X).
    \end{equation*}
    Then the sequence $\{\mu_j\}_{j=1}^{\infty}$ is convergent to $\lambda$.
    It follows that $\lambda\in\sigma(T+X)$.
    Therefore, $\sigma(T+X)=\overline{\mathbb{D}}$.
    This completes the proof.
\end{proof}

\subsection{Perturbations in \texorpdfstring{$C^*$}{C-star}-algebras of real rank zero}\label{subsec perturbation-RR0}

Inspired by \Cref{lem essential=dense} and \Cref{lem small-TE}, we obtain an analogous technical lemma in unital $C^*$-algebras of real rank zero, by which we prove \Cref{thm main-RR0}.

\begin{lemma}\label{lem small-TE-RR0}
    Let $\mathcal{A}$ be a unital $C^*$-algebra of real rank zero and $\mathcal{I}$ an essential ideal in $\mathcal{A}$.
    Suppose that $T$ is an element in $\mathcal{A}$ and $\lambda$ is a boundary point of $\sigma(T)$.
    Then for any $\varepsilon>0$, there exists a nonzero projection $E$ in $\mathcal{I}$ such that $\|(T-\lambda I)E\|<\varepsilon$.
\end{lemma}

\begin{proof}
    Without loss of generality, we assume that $\lambda=0$.
    Then $T^*T$ is not invertible in $\mathcal{A}$ by \Cref{rem sigma-ap}.
    Since $\mathcal{A}$ has real rank zero, there are nonzero projections $\{P_j\}_{j=1}^n$ in $\mathcal{A}$ with sum $I$ and scalars $0=\lambda_1\leqslant\cdots\leqslant\lambda_n$ such that
    \begin{equation*}
        \|T^*T-(\lambda_1P_1+\cdots+\lambda_nP_n)\|<\varepsilon^2.
    \end{equation*}
    It follows that $\|P_1T^*TP_1\|<\varepsilon^2$.
    By \Cref{lem E<P-RR0}, there exists a nonzero projection $E$ in $\mathcal{I}$ such that $E\leqslant P_1$.
    It is clear that $\|TE\|<\varepsilon$.
\end{proof}

By \Cref{lem small-TE-RR0} and the proof of \Cref{prop main}, we obtain \Cref{thm main-RR0} for unital $C^*$-algebras of real rank zero.

\begin{proof}[Proof of \Cref{thm main-RR0}]
    Let $\varepsilon_0=\frac{\varepsilon}{2}>0$.
    Following the proof of \Cref{prop main}, we may assume that $0\in\sigma(T)$ and $\operatorname{Re}z\leqslant 0$ for all $z\in\sigma(T)$.
    Moreover, there exists $0<\delta<\varepsilon_0$ such that for every operator $A$ with $\|A-T \|<\delta$ and every $z\in\sigma(A)$, we have $\operatorname{Re}z<\varepsilon_0$.
    By \Cref{lem small-TE-RR0}, there exists a nonzero projection $E$ in $\mathcal{I}$ such that $\|TE\|<\delta$.
    Let $X=(\varepsilon_0I-T)E$.
    Similar to the proof of \Cref{prop main}, we have $\|X\|<\varepsilon$ and the spectrum of $T+X$ is disconnected.
\end{proof}

Next, we briefly discuss \Cref{problem 1} for abelian unital $C^*$-algebras at the end of this paper.
For every subset $X$ of $\mathbb{R}$, we denote by $|X|$ the \emph{cardinality} of $X$, and by $\operatorname{diam}(X) = \sup\{ |x-y| \colon x,y \in X\}$ its \emph{diameter}.

\begin{lemma}\label{lem diam-Xn}
    Let $X$ be a nonempty compact subset of $\mathbb{R}$. If $X$ is not a union of finitely many disjoint closed intervals, then there are nonempty closed and open subsets $\{X_n\}_{n=1}^\infty$ of $X$ such that $\operatorname{diam}(X_n)\to 0$ as $n\to\infty$.
\end{lemma}

\begin{proof}
Since $\mathbb{R}\setminus X$ is open, there are disjoint open intervals $\{(a_n,b_n)\}_{n=1}^{N}$ with union $\mathbb{R}\setminus X$, where $N$ may be infinite.
If $N$ is finite, then there exists an isolated point $x_0$ in $X$ by assumption.
In this case, we can take $X_n=\{x_0\}$ for every $n\geqslant 1$.
Suppose that $N$ is infinite, $(a_1,b_1)=(-\infty,\min X)$, and $(a_2,b_2)=(\max X,\infty)$.
We choose a point $c_n\in(a_n,b_n)$ for each $n\geqslant 1$.
Then $\{c_n\}_{n=1}^{\infty}$ is bounded and it contains a monotone convergent subsequence $\{c_{n_k}\}_{k=1}^{\infty}$.
Without loss of generality, we assume that $\{c_{n_k}\}_{k=1}^{\infty}$ is monotone increasing.
Let $X_k=(c_{n_k},c_{n_{k+1}})\cap X=[c_{n_k},c_{n_{k+1}}]\cap X$ for each $k\geqslant 1$.
Then $X_k$ is a closed and open subset of $X$.
Since $\{c_{n_k}\}_{k\geqslant 1}$ is convergent, we obtain that $\operatorname{diam}(X_k)\leqslant c_{n_{k+1}}-c_{n_k}\to 0$ as $k\to\infty$.
This completes the proof.
\end{proof}

With \Cref{lem diam-Xn}, we prove the following result.

\begin{proposition}\label{prop subset-R}
    Let $X$ be a compact subset of the real line $\mathbb{R}$ with $|X|>1$.
    Then the set of functions in $C(X)$ with disconnected spectrum is dense if and only if $X$ is not a union of finitely many disjoint closed intervals.
\end{proposition}

\begin{proof}
    Suppose that $X$ is a union of finitely many disjoint closed intervals.
    Without loss of generality, we assume that $X=\bigcup_{k=1}^n[2k,2k+1]$ for some $n\geqslant 1$.
    We define a continuous complex-valued function $f$ on $X$ by
    \begin{equation*}
        f(t)=
        \begin{cases}
            2t-5, & t\in[2,3], \\
            i(2t-4k-1), & t\in[2k,2k+1]~\text{ and }~2\leqslant k \leqslant n.
        \end{cases}
    \end{equation*}
    Then $f$ is not a limit of functions in $C(X)$ with disconnected spectrum.

    In the remainder of the proof, suppose that $X$ is not a union of finitely many disjoint closed intervals.
    By \Cref{lem diam-Xn}, there are nonempty closed and open subsets $\{X_n\}_{n=1}^{\infty}$ of $X$ such that $\mathrm{diam}(X_n)\to 0$ as $n\to\infty$.
    By taking subsequence, we may assume that $\mathrm{dist}(t_0,X_n)\to 0$ for some point $t_0\in X$.
    Moreover, we assume that $X\subseteq[-N,N]$ for some large integer $N\geqslant 1$.
    Let $f\in C(X)$ and $\varepsilon>0$.
    By the Tietze extension theorem, there exists a continuous function $F$ on $[-N,N]$ such that $F(t)=f(t)$ for every $t\in X$.
    Clearly, there exists a piecewise linear function $G\in C([-N,N])$ such that $\|G(t)-F(t)\|<\frac{\varepsilon}{3}$ for every $t\in[-N,N]$.
    By assumption, there exists a large integer $n_0\geqslant 1$ such that $|f(t)-f(t_0)|<\frac{\varepsilon}{3}$ for every $t\in X_{n_0}$.
    Since $G$ is piecewise linear, there exists a scalar $\lambda\in\mathbb{C}$ not in the range of $G$ such that $|\lambda-f(t_0)|<\frac{\varepsilon}{3}$.
    We define a function
    \begin{equation*}
        g(t)=
        \begin{cases}
            \lambda, & t\in X_{n_0}, \\
            G(t), & t\in X\setminus X_{n_0}.
        \end{cases}
    \end{equation*}
    Then $g\in C(X)$, $\|g-f\|<\varepsilon$, and the spectrum of $g$ is disconnected.
    This completes the proof.
\end{proof}

\begin{remark}\label{rem subset-C}
Let $X$ be a compact subset of the complex plane $\mathbb{C}$ with $|X|>1$.
If the set of functions in $C(X)$ with disconnected spectrum is dense, then the interior of $X$ is empty.
Otherwise, we may assume that $0$ is an interior point of $X$ and $2\overline{\mathbb{D}}\subseteq X$.
Let
\begin{equation*}
    F(z)=
    \begin{cases}
        z, & z\in\overline{\mathbb{D}}, \\
        z(2-|z|), & z\in 2\overline{\mathbb{D}}\setminus\overline{\mathbb{D}},\\
        0, & |z|>2.
    \end{cases}
\end{equation*}
Let $f$ be the restriction of $F$ on $X$.
Suppose that $g\in C(X)$ and $\|g-f\|<\frac{1}{2}$.
It is clear that
\begin{equation*}
  g(z)\in\frac{1}{2}\overline{\mathbb{D}}\subseteq g(\overline{\mathbb{D}}) \subseteq g(2\overline{\mathbb{D}})\subseteq g(X)
\end{equation*}
for each $z\in X\setminus 2\overline{\mathbb{D}}$.
Thus, $g(X)=g(2\overline{\mathbb{D}})$ is connected.
Therefore, $f$ is not a limit of functions in $C(X)$ with disconnected spectrum.
In general, it seems difficult to find a necessary and sufficient condition on $X$ as in \Cref{prop subset-R}.
\end{remark}

% \section{Acknowledgments}

% \begin{thebibliography}{99}

% \bibitem{AP}
% C. Akemann, G. Pedersen.
% {\em Ideal perturbations of elements in $C^*$-algebras.}
% Math. Scand. 41 (1977), no. 1, 117-139.
% MR0473848

% \bibitem{Con}
% A. Connes.
% {\em Classification of injective factors. Cases $\mathrm{II}_1$, $\mathrm{II}_\infty$, $\mathrm{III}_\lambda$, $\lambda\ne 1$.}
% Ann. of Math. (2) 104 (1976), no. 1, 73-115.
% MR0454659

% \bibitem{FJMSSW}
% J. Fang, C. Jiang, M. Ma, J. Shen, R. Shi, and T. Wang.
% {\em Density of irreducible operators in the trace-class norm.}
% arXiv:2504.17190v5, 2025.

% \bibitem{SS99}
% A. Sinclair, R. Smith.
% {\em Cartan subalgebras of finite von Neumann algebras.}
% Math. Scand. 85 (1999), no. 1, 105-120.
% MR1707749

% \end{thebibliography}

\end{document}